\documentclass[10pt]{article}

\usepackage{amsthm}
\usepackage[ansinew]{inputenc}   
\usepackage[english]{babel}
\usepackage{anysize}
\usepackage{amsfonts}
\usepackage{amssymb}
\usepackage{amsmath}
\usepackage{doc}
\usepackage{exscale}
\usepackage{fontenc}
\usepackage{ifthen}
\usepackage{latexsym}
\usepackage{makeidx}
\usepackage{syntonly}
\usepackage{graphicx}
\usepackage{float}
\usepackage{array}
\usepackage[all]{xy}
\usepackage[usenames,dvipsnames]{color}

\marginsize{2.5cm}{2.5cm}{2.5cm}{2.5cm}

\newtheorem{theorem}{Theorem}[section]
\newtheorem{corollary}[theorem]{Corollary}
\newtheorem{lemma}[theorem]{Lemma}
\newtheorem{proposition}[theorem]{Proposition}
\newtheorem{definition}[theorem]{Definition}
\newtheorem{remark}[theorem]{Remark}

\newtheorem{example}[theorem]{Example}
\newtheorem{lem-df}[theorem]{Lemma-Definition}
\newtheorem{conjecture}[theorem]{Conjecture}
\newtheorem{rmk-df}[theorem]{Remark-Definition}
\newtheorem{ass}[theorem]{Extra assumption}

\newtheorem*{thm-main}{Theorem \ref{main-thm}}
\newtheorem*{cor-main}{Corollary \ref{cor-adj-global}}

\newcommand{\beq}{\begin{equation}}
\newcommand{\enq}{\end{equation}}
\newcommand{\beqn}{\begin{equation*}}
\newcommand{\enqn}{\end{equation*}}

\newcommand{\tq}{\, | \,}

\newcommand{\ra}{\rightarrow}

\newcommand{\Ra}{\Rightarrow}

\newcommand{\hra}{\hookrightarrow}
\newcommand{\longra}{\longrightarrow}
\newcommand{\Longra}{\Longrightarrow}
\newcommand{\longhra}{\ensuremath{\lhook\joinrel\relbar\joinrel\rightarrow}}

\newcommand{\thra}{\twoheadrightarrow}

\newcommand{\mC}{\mathbb{C}}

\newcommand{\mG}{\mathbb{G}}

\newcommand{\mP}{\mathbb{P}}

\newcommand{\caA}{\mathcal{A}}

\newcommand{\caD}{\mathcal{D}}
\newcommand{\caE}{\mathcal{E}}
\newcommand{\caF}{\mathcal{F}}
\newcommand{\caG}{\mathcal{G}}

\newcommand{\caK}{\mathcal{K}}
\newcommand{\caL}{\mathcal{L}}
\newcommand{\caM}{\mathcal{M}}
\newcommand{\caN}{\mathcal{N}}
\newcommand{\caO}{\mathcal{O}}
\newcommand{\caQ}{\mathcal{Q}}
\newcommand{\caS}{\mathcal{S}}
\newcommand{\caT}{\mathcal{T}}

\newcommand{\caW}{\mathcal{W}}
\newcommand{\caX}{\mathcal{X}}

\DeclareMathOperator{\Hom}{Hom}

\DeclareMathOperator{\Ext}{Ext}
\DeclareMathOperator{\caExt}{{\mathcal Ext}}
\DeclareMathOperator{\caHom}{{\mathcal Hom}}

\DeclareMathOperator{\id}{id}

\DeclareMathOperator{\coker}{coker}
\DeclareMathOperator{\im}{im}

\DeclareMathOperator{\Pic}{Pic}

\DeclareMathOperator{\Spec}{Spec}
\DeclareMathOperator{\Supp}{Supp}
\DeclareMathOperator{\rk}{rk}

\begin{document}

\title{On deformations of curves supported on rigid divisors}

\author{V\'{\i}ctor Gonz\'{a}lez-Alonso
\footnote{Previous address: Departament de Matem\`atica Aplicada I, Universitat Polit\`ecnica de Catalunya (UPC-BarcelonaTECH), Av. Diagonal 647, 08028 Barcelona, Spain.}
\thanks{The author developed this work with the support of the Spanish ``Ministerio de Econom\'{\i}a y Competitividad'' (through the project MTM2012-38122-C03-01/FEDER) and the ``Generalitat the Catalunya'' (through the project 2009-SGR-1284). He was also partially supported by the grant FPU-AP2008-01849 of the Spanish ``Ministerio de Educaci\'{o}n'', and by the ERC StG 279723 ``Arithmetic of algebraic surfaces'' (SURFARI).}\\
\small{Institut f\"ur Algebraische Geometrie, Leibniz Universit\"at Hannover}\\
\small{Welfengarten 1, 30167 Hannover, Germany} \\
\small{gonzalez@math.uni-hannover.de}
}

\maketitle

\begin{abstract}
Motivated by a conjecture of Xiao, we study supporting divisors of fibred surfaces. On the one hand, after developing a formalism to treat one-dimensional families of varieties of any dimension, we give a structure theorem for fibred surfaces supported on relatively rigid divisors. On the other hand, we study how to produce supporting divisors by constructing a {\em global adjoint map} for a fibration over a curve (generalizing the infinitesimal constructions of Collino, Pirola, Rizzi and Zucconi).
\end{abstract}

\section{Introduction}

\label{sect-intro}

The study of fibrations of algebraic varieties, or more generally, of flat families and deformations, is quite a general problem in algebraic geometry. From the infinitesimal point of view, there is a well understood deformation theory, with the Kodaira-Spencer map playing a central role. Another useful tool to study first-order deformations are {\em supporting divisors}, which roughly speaking encode where in the variety the deformation is taking place, and can be somehow used to measure how far a given deformation is from being trivial. In the case of irregular varieties, supporting divisors are related to {\em adjoint images} (or {\em maps}), introduced by Collino and Pirola \cite{ColPir} for curves, extended later to higher dimensions by Pirola and Zucconi \cite{PZ}, and to higher-dimensional bases (but still with curves as fibres) by Pirola and Rizzi \cite{Pir-Riz}.

In this article we focus on families of varieties over smooth curves, generalizing both the notion of supporting divisor and the construction of adjoint maps to the non-infinitesimal case. Our initial motivation was the following conjecture about the relative irregularity of non-isotrivial fibrations.
\begin{conjecture}[modified Xiao's conjecture] \label{final-conj}
For any non-trivial fibration $f: S \ra B$ with fibres of genus $g$ and relative irregularity $q_f = q\left(S\right) - g\left(B\right)$, one has
\beqn
q_f \leq \frac{g}{2}+1.
\enqn
\end{conjecture}
It is known (see the Appendix by Beauville to \cite{Deb-Beau}) that any fibration verifies $0 \leq q_f \leq g$, with $q_f = g$ if and only if it is trivial. For isotrivial but not trivial fibrations the stronger bound
\beq \label{original-ineq}
q_f \leq \frac{g+1}{2}
\enq
can be deduced from several results of Serrano in \cite{Ser-Iso}. Considering non-isotrivial fibrations, Xiao proved in \cite{Xiao-P1} the same inequality (\ref{original-ineq}) when the base curve is $B = \mP^1$, and conjectured it to hold for any base. However, a few years later, Pirola proved in \cite{Pir-Xiao} the existence of a counterexample with $q_f=3,g=4$, and proposed the above modified version. Furthermore, in the recent work \cite{Alb-Pir}, Albano and Pirola produce more counterexamples to the original conjecture with different values of $q_f$ and $g$, but still satisfying $q_f = \frac{g}{2}+1$. It is worth to note that all these examples are obtained after proving the existence of special curves in certain moduli spaces of coverings of curves, and further studying these examples is therefore a very difficult task. For example, except in one of the cases, not even $q\left(S\right)$ is known.

The general case of Conjecture \ref{final-conj} is proved by the author in a forthcoming joint work with Barja and Naranjo \cite{Xiao-2}, using some results of the present article. Although our main interest is on families of curves, most of the formalism carries over without many problems to one-dimensional families of varieties of any dimension. Since our techniques may also be useful to study deformations of higher-dimensional varieties, we have developed as many results as possible in this general setting. There are, however, specific results for families of curves that have no obvious higher-dimensional analogue.

The paper is divided into two main sections. In the first one we introduce supporting divisors in the non-infinitesimal case, giving a language to study one-dimensional families of varieties of any dimension. The main result in this section is the following theorem about fibred surfaces.
\begin{thm-main}
Let $S$ be a compact surface, and $f: S \ra B$ a stable fibration by curves of genus $g$ and relative irregularity $q_f =q\left(S\right) - g\left(B\right) \geq 2$. Suppose $f$ is supported on an effective divisor $D$ without components contained in fibres. Suppose also that $D \cdot C \leq 2g\left(C\right)-2-C^2$ for any component $C$ of a fibre, and that $h^0\left(F,\caO_F\left(D_{|F}\right)\right) = 1$ for some smooth fibre $F$. Then there is another fibration $h: S \ra B'$ over a curve of genus $g\left(B'\right) = q_f$. In particular $S$ is a covering of the product $B \times B'$, and both surfaces have the same irregularity.
\end{thm-main}
For us, a supporting divisor for a family $f: \caX \ra B$ is a divisor $\caD$ on the total space $\caX$ whose restriction to a general fibre supports the corresponding infinitesimal deformation. This means that the pull-back sequence $\xi_{\caD}$ below (where $\caL_{\caD} = \ker(\Omega_{\caX/B}^1 \ra \Omega_{\caX/B|\caD}^1)$).
\beqn
\xymatrix{
\xi_{\caD}: & 0 \ar[r] & f^*\omega_B \ar[r] \ar@{=}[d] & \caF_{\caD} \ar[r] \ar@{^(->}[]+<0mm,-3mm>;[d] & \caL_{\caD} \ar[r] \ar@{^(->}[]+<0mm,-3mm>;[d] & 0 \\
\xi: & 0 \ar[r] & f^*\omega_B \ar[r] & \Omega_{\caX}^1 \ar[r] & \Omega_{\caX/B}^1 \ar[r] & 0
}
\enqn
is split around a general fibre, but it does not necessarily split globally. Fortunately, if the fibres are curves and the divisor has low degree with respect to the components of the fibres, then the local splitting implies the global splitting (Corollary \ref{cor-converse-split}). Using this global splitting, it is just a matter of technicalities to prove Proposition \ref{prop-conjunta}, which shows that (up to change of base) we can replace the splitting subsheaf $\caL_D$ by a line bundle with much better properties. After these results, the proof of Theorem \ref{main-thm} follows easily with the help of the classical Castelnuovo-de Franchis theorem.

In the second section we consider the problem of finding supporting divisors for a given family. We first consider the infinitesimal case, introducing the {\em adjoint bundle} of a variety (Definition \ref{df-adj-bundle}) and computing its top Chern class. This gives a numerical condition on a variety which is sufficient for the existence of a subspace with vanishing adjoint image (Theorem \ref{thm-adj-inft}). Then we construct, considering only curves as fibres, a {\em global adjoint map}, which glues the adjoint maps on the smooth fibres in a coherent way (covering also the singular fibres). Putting together this construction and the numerical condition mentioned above, we obtain Theorem \ref{pro-adj-global} and its following
\begin{cor-main}
If $f: S \ra B$ is a fibration of genus $g$ such that $q_f > \frac{g+1}{2}$, then after a base change $B' \ra B$ the new fibration $f': S' \ra B'$ is supported on a divisor $D$ such that $D \cdot F < 2g-2$ for any fibre $F$. Furthermore, if $f$ is relatively minimal with reduced fibres, then $D \cdot C \leq 2g\left(C\right)-2-C^2$ for any component $C$ of a fibre.
\end{cor-main}
That is, if the relative irregularity of the fibration is too big (in particular, it does not verify Xiao's conjecture), then there is a divisor satisfying the first hypothesis in Theorem \ref{main-thm} (up to change of base).
In this way, both sections of the paper get connected, and its relation with Xiao's conjecture is evident.


\section{One-dimensional families of varieties}

\label{sect-deform}

We will first recall a basic definition which is the starting point of the whole work. Let $X$ be a smooth compact complex variety of dimension $d$, and $\xi \in H^1\left(X,T_X\right)$ the Kodaira-Spencer class of a first-order infinitesimal deformation $X$, given by the extension of vector bundles
\beqn
\xi: \quad 0 \longra \caO_X \longra \caF \longra \Omega_X^1 \longra 0.
\enqn

\begin{definition} \label{df-deform-infinit}
The deformation $\xi$ is said to be {\em supported on} an effective divisor $D$ of $X$ if
\beqn
\xi \in K_D = \ker\left(H^1\left(X,T_X\right) \ra H^1\left(X,T_X\left(D\right)\right)\right) = \im\left(H^0\left(D,T_X\left(D\right)_{|D}\right) \ra H^1\left(X,T_X\right)\right),
\enqn
or equivalently, if the pull-back sequence in the following diagram splits.
\begin{equation} \label{pull-back-fibre}
\xymatrix{
\xi_D: & 0 \ar[r] & \caO_X \ar[r] \ar@{=}[d] & \caF_D \ar[r] \ar@{^(->}[]+<0mm,-3mm>;[d] & \Omega_X^1\left(-D\right) \ar[r] \ar@{^(->}[]+<0mm,-3mm>;[d] & 0 \\
\xi: & 0 \ar[r] & \caO_X \ar[r] & \caF \ar[r] & \Omega_X^1 \ar[r] & 0
}
\end{equation}
\end{definition}

Intuitively, one can think of a deformation supported on $D$ as an infinitesimal deformation of the complex structure of $X$ that does not change $X-D$.

\begin{example}[Schiffer variations]
In the case $X=C$ is a curve and $\xi$ is supported on a single point $p$, $\xi$ is said to be a {\em Schiffer variation}. They were introduced by Spencer and Schiffer in \cite{SpeSchi}, and they proved very useful to study the Griffiths infinitesimal invariant of a deformation (see \cite{Gri-IVHS3}).
\end{example}

\begin{remark}
Supporting divisors always exist. In fact, for $D$ ample enough it holds $H^1\left(X,T_X\left(D\right)\right)=0$, and hence $D$ supports any infinitesimal deformation. Therefore the interesting supporting divisors are those such that $H^1\left(X,T_X\left(D\right)\right) \neq 0$, whose existence is not clear at first sight. If $X$ is irregular, the Adjoint Theorem (Theorem \ref{thm-adj}, to be introduced in the next section) provides a construction of such supporting divisors as base divisors of certain subsystems of the canonical system of $X$.
\end{remark}

\begin{remark}
Since ample enough divisors support any deformation, one cannot expect any meaningful relation between any supporting divisor and the geometry of $X$. However, if $X$ has maximal Albanese dimension and a supporting divisor $D$ is obtained by means of the Adjoint Theorem, then it follows from \cite{PZ} that $D$ contains the ramification divisor of the Albanese map of $X$.
\end{remark}

The aim of this section is to develop and study an analogous construction in the case of a smooth variety of dimension $d+1$ fibred over a curve (that is, a one-dimensional family of $d$-dimensional varieties), extending them also to possibly singular fibres. Some of the ideas used here also appear in \cite{Sernesi}.

\subsection{General setting}

\label{subsect-deform-general}

From now on, let $f: \caX \ra B$ be a proper morphism of smooth complex varieties, where $B$ is a smooth curve (not necessarily compact), and $\caX$ has dimension $d+1$. It can be considered as a family of $d$-dimensional complex spaces $X_b = \caX \times_B \Spec\mC\left(b\right)$, $b \in B$. Let $B^o \subseteq B$ be the open set of regular values, so that $X_b$ is smooth if and only if $b \in B^o$. Denote also by $\caX^o = f^{-1}\left(B^o\right)$ the union of the smooth fibres. We will assume that $f$ is {\em non-isotrivial}.

For every smooth fibre $X=X_b$, $f$ induces an infinitesimal deformation, whose Kodaira-Spencer class $\xi_b \in H^1\left(X,T_X\right) \otimes T_{B,b}^{\vee} \cong \Ext_{\caO_X}^1\left(\Omega_X^1,\caO_X \otimes T_{B,b}^{\vee}\right)$ is the extension class of
\beqn
0 \longra N_{X/\caX}^{\vee} = \caO_X \otimes T_{B,b}^{\vee} \longra \Omega_{\caX|X}^1 \longra \Omega_X^1 \longra 0,
\enqn
obtained by restricting the sequence
\beq \label{eq-ext-global}
\xi: \quad 0 \longra f^* \omega_B \longra \Omega_{\caX}^1 \longra \Omega_{\caX/B}^1 \longra 0
\enq
defining the {\em sheaf of relative differentials} $\Omega_{\caX/B}^1$. Indeed, $\Omega_{\caX/B}^1$ is locally free of rank $d$ on $\caX^o$, and restricts to the cotangent bundle of the smooth fibres. Its dual is
\beqn
T_{\caX/B} = \caHom_{\caO_{\caX}}\left(\Omega_{\caX/B}^1,\caO_{\caX}\right) = \ker\left(T_{\caX} \longra f^*T_B\right),
\enqn
the {\em relative tangent sheaf}. It is also locally free of rank $d$ on $\caX^o$, and restricts to the tangent bundle of the smooth fibres. We will also consider $\omega_{\caX/B} = \omega_{\caX} \otimes f^*\omega_B^{\vee}$, the relative dualizing sheaf, a line bundle on $\caX$ that restricts to the canonical bundle of the fibres.

\begin{remark}
Usually, one chooses a generator of $T_{B,b}$ and considers $\xi_b \in H^1\left(X,T_X\right)$. However, this might not be done globally on $B$, so we will keep the natural twist by $T_{B,b}^{\vee}$.
\end{remark}

Since the fibration $f$ is not isotrivial, $\xi_b \neq 0$ for general $b \in B^o$. Furthermore, if $\caD$ is any effective divisor on $\caX$, we can also ask whether $\xi_b$ is supported on $D_b = \caD_{|X_b}$, and if the answer is positive, what consequences for the fibration $f$ does it have.

We aim now for a globalization over $B$ of the vector space $H^1\left(X,T_X\right)$. The first candidate one can think of is the sheaf $R^1f_*T_{\caX/B}$, but it is difficult to track its behaviour at the critical values. Instead, we consider the forthcoming sheaf $\caE$, which is closely related to $R^1f_*T_{\caX/B}$ (see Lemma \ref{lem-inj-R1TE}) but behaves much better around the non-smooth fibres.

\begin{definition}
Let $\caE$ be the sheaf on $B$ defined as (see the Appendix for the definition and main properties of the relative $\caExt$ sheaves)
\beqn
\caE = \caExt_f^1\left(\Omega_{\caX/B}^1,f^*\omega_B\right) \cong \caExt_f^1\left(\Omega_{\caX/B}^1,\caO_{\caX}\right)\otimes\omega_B.
\enqn
\end{definition}

\begin{lemma} \label{lem-inj-R1TE}
There is an injection
\beqn
R^1f_*T_{\caX/B} \otimes \omega_B \longhra \caE
\enqn
which is an isomorphism over $B^o$. In particular, for a general regular value $b \in B^o$ it holds
\beqn
\caE \otimes \mC\left(b\right) \cong H^1\left(X_b,T_{X_b}\right) \otimes T_{B,b}^{\vee}.
\enqn
\end{lemma}
\begin{proof}
The injection is obtained directly from the beginning of the five-term exact sequence corresponding to the local-global spectral sequence (Theorem \ref{thm-ext}.4)
\beqn
R^pf_*\caExt_{\caO_{\caX}}^q\left(\Omega_{\caX/B}^1,f^*\omega_B\right) \Longra \caExt_f^{p+q}\left(\Omega_{\caX/B}^1,f^*\omega_B\right),
\enqn
combined with the projection formula. The cokernel is $f_*\caExt^1_{\caO_{\caX}}\left(\Omega_{\caX/B}^1,f^*\omega_B\right)$, which is zero on $B^o$ because $\Omega_{\caX/B}^1$ is locally free on $\caX^o$, and the first claim follows.

As for the statement about the general regular values, note that $T_{\caX/B|X_b} = T_{X_b}$ for any smooth fibre, and the base-change map $\caE \otimes \mC(b) \ra H^1\left(X_b,T_{X_b}\right) \otimes T_{B,b}^{\vee}$ is an isomorphism over the open set of $B$ where the function $b \mapsto h^1\left(X_b,T_{X_b}\right)$ is constant.
\end{proof}

The family $f$ naturally defines a section of $\caE$ as follows: the extension class of (\ref{eq-ext-global}) defines an element $\xi \in \Ext^1_{\caO_{\caX}}\left(\Omega_{\caX/B}^1,f^*\omega_B\right)$, and the spectral sequence
\beq \label{eq-ss-1}
E_2^{p,q} = H^p\left(B,\caExt_f^q\left(\Omega_{\caX/B}^1,f^*\omega_B\right)\right) \Longra \Ext_{\caO_{\caX}}^{p+q}\left(\Omega_{\caX/B}^1,f^*\omega_B\right)
\enq
gives a map
\beq \label{map-rho}
\rho: \Ext_{\caO_{\caX}}^1\left(\Omega_{\caX/B}^1,f^*\omega_B\right) \longra H^0\left(B,\caExt_f^1\left(\Omega_{\caX/B}^1,f^*\omega_B\right)\right) = H^0\left(B,\caE\right).
\enq
By construction, the section $\rho\left(\xi\right)$ of $\caE$ maps a general value $b \in B^o$ to the Kodaira-Spencer class $\xi_b$ of the deformation of $X_b$.

We move now to the definition of supporting divisors of the family $f$.

\begin{definition} \label{df-global-supp}
We say that the the family $f$ {\em is supported on} an effective divisor $\caD \subset \caX$ if, for general $b \in B$, the infinitesimal deformation $\xi_b$ is supported in $\caD_{|X_b}$.
\end{definition}

\begin{remark} \label{rmk-same-support}
Since Definition \ref{df-global-supp} takes into account only general fibers, it follows that if $\caD,\caD' \subset \caX$ are two effective divisors with exactly the same components dominating $B$, then $\xi$ is supported on $\caD$ if and only if it is supported on $\caD'$. Therefore we will often assume that a supporting divisor $\caD$ has no components contracted by $f$.
\end{remark}

Since the definition is local around general fibres, the next Lemma follows immediately.

\begin{lemma} \label{lem-supp-base-change}
Let $p: B' \ra B$ be a finite morphism and $\caX'$ a desingularization of $\caX \times_B B'$. Denote by $f': \caX' \ra B'$ and $p': \caX' \ra \caX$ the two natural projections. Then $f$ is supported on a divisor $\caD$ if and only if $f'$ is supported on $\caD' = \left(p'\right)^* \caD$.
\end{lemma}

Our next purpose is to relate Definition \ref{df-global-supp} to the splitting of a certain pull-back of
\beqn
\xi: \quad 0 \longra f^*\omega_B \longra \Omega_{\caX}^1 \longra \Omega_{\caX/B}^1 \longra 0,
\enqn
in the spirit of the equivalence in Definition \ref{df-deform-infinit}. To this aim we introduce the following sheaves, which play the role of $\Omega_X^1\left(-D\right)$ and $H^1\left(X,T_X\left(D\right)\right)$ in the infinitesimal setting.

\begin{definition} \label{df-aux-supp}
For any effective divisor $\caD \subset \caX$, considered as a closed subscheme, define
\beqn
\caL_{\caD} = \ker\left(\Omega_{\caX/B}^1 \ra \Omega_{\caX/B|\caD}^1\right) \quad \text{and} \quad \caE_{\caD} = \caExt_f^1\left(\caL_{\caD},f^*\omega_B\right) = \caExt_f^1\left(\caL_{\caD},\caO_{\caX}\right) \otimes \omega_B.
\enqn
\end{definition}

The proof of the following Lemma is straightforward (in fact, the proof of the first statement is completely analogous to that of Lemma \ref{lem-inj-R1TE}).

\begin{lemma} \label{lem-inj-R1TDE}
There is a natural inclusion
\beqn
R^1f_*\left(T_{\caX/B}\left(\caD\right)\right) \otimes \omega_B \longhra \caE_{\caD}
\enqn
which is an isomorphism over $B^o$. Furthermore, the inclusion $\iota: \caL_{\caD} \hra \Omega_{\caX/B}^1$ induces a natural map of sheaves $\caE \ra \caE_{\caD}$ which over a general regular value $b \in B^o$ coincides with
$$H^1\left(X_b,T_{X_b}\right) \otimes T_{B,b}^{\vee} \ra H^1\left(X_b,T_{X_b}\left(D_b\right)\right) \otimes T_{B,b}^{\vee}.$$
\end{lemma}

By definition, $f$ is supported on $\caD$ if the pull-back sequence
\beq \label{pull-back-seq}
\xymatrix{
\xi_{\caD}: & 0 \ar[r] & f^*\omega_B \ar[r] \ar@{=}[d] & \caF_{\caD} \ar[r] \ar@{^(->}[]+<0mm,-3mm>;[d] & \caL_{\caD} \ar[r] \ar@{^(->}[]+<0mm,-3mm>;[d] & 0 \\
\xi: & 0 \ar[r] & f^*\omega_B \ar[r] & \Omega_{\caX}^1 \ar[r] & \Omega_{\caX/B}^1 \ar[r] & 0
}
\enq
splits around a general fibre. Of course, this does not imply in general that the pull-back sequence is itself split. In order to find conditions under which the pull-back sequence (\ref{pull-back-seq}) splits, consider the following commutative diagram
\beqn
\xymatrix{
\Ext_{\caO_{\caX}}^1\left(\Omega_{\caX/B}^1,f^*\omega_B\right) \ar[r]^-{\iota^*} \ar[d]_{\rho} & \Ext_{\caO_{\caX}}^1\left(\caL_{\caD},f^*\omega_B\right) \ar[d]^{\rho_{\caD}} \\
H^0\left(B,\caE\right) \ar[r]^-{\widetilde{\iota^*}} & H^0\left(B,\caE_{\caD}\right)
}
\enqn
where the vertical maps are given by the corresponding local-global spectral sequences, and the horizontal ones are induced by the inclusions of sheaves mentioned in Lemma \ref{lem-inj-R1TDE}.

On the one hand, at a general regular value $b \in B^o$ we have $\widetilde{\iota^*}\left(\rho\left(\xi\right)\right)\left(b\right) = \xi_{b,D_b}$. Hence it is clear that $f$ is supported on $\caD$ if and only if
\beqn
\widetilde{\xi_{\caD}} := \widetilde{\iota^*}(\rho(\xi)) = \rho_{\caD}(\iota^*(\xi)) \in H^0\left(B,\caE_{\caD}\right)
\enqn
is a torsion section of $\caE_{\caD}$.

On the other hand, the global splitting of (\ref{pull-back-seq}) is equivalent to the vanishing of
\beqn
\xi_{\caD} = \iota^*\left(\xi\right) \in \Ext_{\caO_{\caX}}^1\left(\caL_{\caD},f^*\omega_B\right).
\enqn

Therefore, to deduce the vanishing of $\xi_{\caD}$ from the fact that $\widetilde{\xi_{\caD}} = \rho_{\caD}(\xi_{\caD})$ is torsion, there are two things we must consider:
\begin{itemize}
\item the torsion of $\caE_{\caD}$ (ideally, we want it to be torsion-free), and
\item the kernel of $\rho_{\caD}$, which by the local-global spectral sequence is
\beqn
\ker \rho_{\caD} = H^1\left(B,f_*\caHom_{\caO_{\caX}}\left(\caL_{\caD},f^*\omega_B\right)\right).
\enqn
\end{itemize}
Dealing with these questions in arbitrary dimensions is quite complicated, but they become easier if we restrict ourselves to families of curves, as we will do now.

\subsection{The special case of fibred surfaces}

\label{subsect-deform-surf}

From now until the end of the section, we will asume that $\caX = S$ is a surface. We will denote the fibres by $F_b$, and $g$ will stand for the genus of the general (smooth) fibres.

In this particular case it is not difficult to control the kernel of $\rho_D$ and the torsion-freeness of $\caE_D$ for a supporting divisor $D$. The former is in fact rather immediate, and for the latter we will need to assume that the fibration is stable (i.e, relatively minimal with reduced and simple normal crossing singular fibres). It is also very useful to control the relation between the sheaves $\Omega_{S/B}^1$ and $\omega_{S/B}$, which one can expect to be very similar since they coincide on the smooth fibres. More precisely, there is a natural map $\alpha: \Omega_{S/B}^1 \ra \omega_{S/B}$ which is not available in the general setting above of higher-dimensional fibres.

We will first focus on the kernel of $\rho_D$.

\begin{lemma} \label{lem-rho-iso}
Let $D \subseteq S$ be an effective divisor such that $D \cdot F_b < 2g-2$. Then $f_*\caL_D^{\vee}=0$, and therefore the map
\beqn
\rho_D: \Ext_{\caO_S}^1\left(\caL_D,f^*\omega_B\right) \longra H^0\left(B,\caE_D\right)
\enqn
is an isomorphism.
\end{lemma}
\begin{proof}
The dual $\caL_D^{\vee}$ is torsion-free, so $f_*\caL_D^{\vee}$ is also torsion-free, hence a vector bundle. Therefore for a general $b \in B$ we have
\beqn
\left(f_*\caL_D^{\vee}\right) \otimes \mC(b) = H^0\left(F_b,\caL_{D|F_b}^{\vee}\right) = H^0\left(F_b,T_{F_b}(D_b)\right) = 0,
\enqn
since $\caL_{D|F_b} = \ker\left(\omega_{F_b} \ra \omega_{F_b|D_b}\right) = \omega_{F_b}(-D_b)$ and the hypothesis $D \cdot F_b < 2g-2$ is equivalent to $\deg\left(T_{F_b}(D_b)\right)<0$. Now, the five-term exact sequence associated to the spectral sequence
\beqn
H^p\left(B,\caExt_f^q\left(\caL_D,f^*\omega_B\right)\right) \Longra \Ext^{p+q}_{\caO_S}\left(\caL_D,f^*\omega_B\right)
\enqn
gives
\beqn
\ker \rho_D = H^1\left(B,f_*\caHom_{\caO_S}\left(\caL_D,f^*\omega_B\right)\right) = H^1\left(f_*\caL_D^{\vee} \otimes \omega_B\right) = 0
\enqn
because $f_*\caL_D^{\vee}=0$, and since $\dim B =1$ we also have
\beqn
\coker \rho_D \subseteq H^2\left(B,f_*\caHom_{\caO_S}\left(\caL_D,f^*\omega_B\right)\right) = 0.
\enqn
\end{proof}

We study now the relation between $\Omega_{S/B}^1$ and $\omega_{S/B}$, which is crucial to prove the torsion-freeness of $\caE_D$ (Proposition \ref{prop-tors-free}), as well as Proposition \ref{prop-conjunta} and therefore Theorem \ref{main-thm}.

\begin{definition}
The {\em Jacobian ideal sheaf} of $f$ is
\beqn
J:= \im\left(T_f: T_S \longra f^*T_B\right) \otimes f^*\omega_B \subseteq \caO_S.
\enqn
It is the ideal of a subscheme $Z$ supported on the critical points of $f$. Denote by $Z_d$ the union of the divisorial components of $Z$, and by $Z_p$ the residual subscheme supported on points.
\end{definition}

\begin{remark}
In \cite{SerDuk}, Serrano defined a sheaf (also denoted by $J$) which is essentially our Jacobian ideal sheaf, but without the twisting by $f^*\omega_B$.
\end{remark}

\begin{lemma}[\cite{SerDuk} Lemma 1.1]
Let $\left\{E_i\right\}$ be the irreducible components of the singular fibers of $f$, and let $\nu_i$ be the multiplicity of $E_i$ as a component of the corresponding fibre. Then:
\begin{enumerate}
\item The relative tangent sheaf $T_{S/B}$ is an invertible sheaf, whose inverse is
\beqn
\left(\Omega_{S/B}^1\right)^{\vee\vee} = T_{S/B}^{\vee} \cong \omega_{S/B}\left(-\sum_i \left(\nu_i-1\right)E_i\right).
\enqn
\item $J^{\vee\vee} \cong \caO_S\left(-\sum_i \left(\nu_i-1\right)E_i\right)$. Therefore $Z_{d}=\sum_i \left(\nu_i-1\right)E_i$ and
\beqn
J =J^{\vee\vee} \otimes I_{Z_p} = I_{Z_p}\left(-\sum_i \left(\nu_i-1\right)E_i\right).
\enqn
\end{enumerate}
\end{lemma}

The following Lemma shows the explicit relation between $\Omega_{S/B}^1$ and $\omega_{S/B}$.

\begin{lemma} \label{lem-versus}
The sheaves $\Omega_{S/B}^1$ and $\omega_{S/B}$ fit into the exact sequence
\beqn
0 \longra \left(f^*\omega_B\left(Z_d\right)\right)_{|Z_d} \longra \Omega_{S/B}^1 \stackrel{\alpha}{\longra} \omega_{S/B} \longra \omega_{S/B|Z} \longra 0.
\enqn
In particular, if $f$ has reduced fibres, then $Z=Z_p$, the map $\alpha$ is injective and $\Omega^1_{S/B} \cong \omega_{S/B}\otimes J$ is torsion-free. In general, $\omega_{S/B}\otimes J$ is the quotient of $\Omega_{S/B}^1$ by its torsion subsheaf.
\end{lemma}
\begin{proof}
Let us first construct the map $\alpha$. Twisting the exact sequence defining $\Omega_{S/B}^1$ by $f^*\omega_B$, one gets
\beqn
0 \longra \left(f^*\omega_B\right)^{\otimes2} \longra f^*\omega_B \otimes \Omega_S^1 \longra f^*\omega_B \otimes \Omega_{S/B}^1 \longra 0.
\enqn
Wedge product gives a map $\widetilde{\beta}: f^*\omega_B \otimes \Omega_S^1 \ra \omega_S$ vanishing on $\left(f^*\omega_B\right)^{\otimes2}$. Therefore, $\widetilde{\beta}$ induces a map $\beta: f^*\omega_B \otimes \Omega_{S/B}^1 \ra \omega_S$. The wanted $\alpha$ is precisely $\beta$ twisted by $\left(f^*\omega_B\right)^{\vee}$. Denoting by $\widetilde{\alpha}$ the corresponding twist of $\widetilde{\beta}$, we get the following diagram with exact rows
\beqn
\xymatrix{
0 \ar[r] & f^*\omega_B \ar[r] \ar[d] & \Omega_S^1 \ar[r] \ar[d]_{\widetilde{\alpha}} & \Omega_{S/B}^1 \ar[r] \ar[d]^{\alpha} & 0 \\
 & 0 \ar[r] & \omega_{S/B} \ar[r] & \omega_{S/B} \ar[r] & 0
}
\enqn
and the snake lemma gives $\coker\alpha = \coker\widetilde{\alpha}$ and $\ker\alpha = \left(\ker\widetilde{\alpha}\right)/f^*\omega_B$. But $\widetilde{\alpha}$ is exactly the tangent map $T_f$ twisted by $\omega_S$, so by the definition of $J$, $Z$ and $Z_d$ we get
\beqn
\ker \widetilde{\alpha} = f^*\omega_B\left(Z_d\right) \quad \text{and} \quad \coker \widetilde{\alpha} = \omega_{S/B|Z}.
\enqn
To conclude, note that $f^*\omega_B \hookrightarrow \ker\widetilde{\alpha}$ is induced by the natural map $\caO_S \hookrightarrow \caO_S\left(Z_d\right)$, so
\beqn
\ker\alpha = f^*\omega_B \otimes \left(\caO_S\left(Z_d\right)/\caO_S\right) = \left(f^*\omega_B\left(Z_d\right)\right)_{|Z_d}.
\enqn
\end{proof}

We can now proceed to study when $\caE_D$ is torsion-free.

\begin{proposition} \label{prop-tors-free}
Assume that the fibration $f$ is stable, and let $D$ be an effective divisor such that $D \cdot C \leq 2g\left(C\right)-2-C^2$ for any component of a fibre, with strict inequality if $C=F_b$ is a smooth fibre. Then the sheaf $\caE_D = \caExt_f^1\left(\caL_D,f^*\omega_B\right)$ is torsion-free.
\end{proposition}
\begin{proof}
Since $\caE_D = \caExt_f^1\left(\caL_D,\caO_S\right) \otimes \omega_B$, it is enough to prove that $\caExt_f^1\left(\caL_D,\caO_S\right)$ is locally free. Since $f$ is stable, the critical locus $Z$ consists of reduced points. Therefore $\Omega_{S/B}^1$ and $\caL_D$ are torsion-free, hence flat over $B$, and Theorem \ref{thm-base-change-ext} (base change for $\caExt_f^i$) can be applied. Furthermore, in this situation we have $\caL_D = \omega_{S/B}\left(-D\right) \otimes J$, and using local coordinates it is easy to prove that for any curve $C$ in $S$ we have the following exact sequence
\beq \label{ex-seq-LD-fibre}
0 \longra \caO_{Z \cap C} \longra \caL_{D|C} \longra \omega_{S/B}\left(-D\right)_{|C} \longra \caO_{Z \cap C} \longra 0.
\enq

We will now prove that for any $b \in B$ the base-change maps
$$\tau^i\left(b\right): \caExt_f^i\left(\caL_D,\caO_S\right) \otimes \mC(b) \ra \Ext^i_{\caO_{F_b}}\left(\caL_{D|F_b},\caO_{F_b}\right)$$
are isomorphisms for $i=2$. In fact we will see that both sides vanish. On the one hand, since $\caL_D$ is torsion-free, we have $\caExt_{\caO_S}^q\left(\caL_D,\caO_S\right)=0$ for $q \geq 2$, and $\caExt_{\caO_S}^1\left(\caL_D,\caO_S\right)$ is supported on $Z$. Using the local-global spectral sequence it follows that $\caExt_f^i\left(\caL_D,\caO_S\right)=0$ for $i \geq 2$. On the other hand, since $\caL_D$ is locally free on the smooth points of $F_b$, the sheaves $\caExt_{\caO_{F_b}}^i\left(\caL_{D|F_b},\caO_{F_b}\right)$ are supported on the singular points of $F_b$ for $i \geq 1$. Moreover, since the singularites of $F_b$ are ordinary nodes, a local computation shows that $\caExt_{\caO_{F_b}}^i\left(\caL_{D|F_b},\caO_{F_b}\right)=0$ for $i \geq 2$. Finally the local-global spectral sequence on $F_b$ gives $\Ext_{\caO_{F_b}}^2\left(\caL_{D|F_b},\caO_{F_b}\right) = 0$.

Now Theorem \ref{thm-base-change-ext} with $i=2$ implies that $\tau^1\left(b\right)$ is surjective for all $b \in B$. Using the same theorem with $i=1$, we get that $\caE_D$ is locally free if and only if $\tau^0\left(b\right)$ is surjective for all $b \in B$. Since by Lemma \ref{lem-rho-iso} it holds $\caExt_f^0\left(\caL_D,\caO_S\right) = f_*\caL_D^{\vee} = 0$, we only need to prove that
$$\Hom_{\caO_{F_b}}\left(\caL_{D|F_b},\caO_{F_b}\right) = 0,$$
which follows from the condition on the $D \cdot C$ for all the components $C$ of $F_b$.
\end{proof}

We can finally state the final result concerning the global splitting of (\ref{pull-back-seq}).

\begin{corollary} \label{cor-converse-split}
If $D \cdot C \leq 2g\left(C\right)-2-C^2$ for any component $C$ of a fibre, with strict inequality for a general fibre $F_b$, then $f$ is supported on $D$ if and only if the pull-back in (\ref{pull-back-seq}) splits.
\end{corollary}

\begin{remark}
The condition $D \cdot C \leq 2g\left(C\right)-2-C^2$ seems a bit strange at first sight, but it is not so restrictive because every supporting divisor produced by the global adjoint map (Corollary \ref{cor-adj-global}) satisfy it.
\end{remark}

Before going through the proof of Theorem \ref{main-thm}, we need a technical result (Proposition \ref{prop-conjunta}) about inclusions $\caL \hra \Omega_{S/B}^1$ or $\caL \hra \omega_{S/B}$ lifting to $\Omega_S^1$ (see Definition \ref{df-lift} below). It will allow us to improve the properties of a supporting divisor.

\begin{definition} \label{df-lift}
We say that a rank-one subsheaf $\caL$ of $\Omega_{S/B}^1$ (resp. $\omega_{S/B}$) {\em lifts to} $\Omega_S^1$ if the inclusion can be factored as an injection $\caL \hra \Omega_S^1$ followed by the projection $\Omega_S^1 \ra \Omega_{S/B}^1$ (resp. the same projection composed with $\alpha: \Omega_{S/B}^1 \ra \omega_{S/B}$). Equivalently, $\caL \hra \Omega_{S/B}^1$ lifts to $\Omega_S^1$ if the pull-back row in the following diagram is split.
\beqn
\xymatrix{
\xi_{\caL} : & 0 \ar[r] & f^*\omega_B \ar[r] \ar@{=}[d] & \caF_{\caL} \ar[r] \ar@{^(->}[]+<0mm,-3mm>;[d] & \caL \ar[r] \ar@{^(->}[]+<0mm,-3mm>;[d] & 0 \\
\xi : & 0 \ar[r] & f^*\omega_B \ar[r] & \Omega_S^1 \ar[r] & \Omega_{S/B}^1 \ar[r] & 0
}
\enqn
\end{definition}

\begin{proposition} \label{prop-conjunta}
Let $f: S \ra B$ be a fibration with reduced fibres, and $E$ an effective divisor without components contracted by $f$. Suppose that $\caL_E \hra \Omega_{S/B}^1$ lifts to $\Omega_S^1$, and also that $E \cdot C \leq 2g\left(C\right)-2-C^2$ for any component $C$ of a fibre, with strict inequality for a smooth fibre $F_b$. Then there exists an effective divisor $D \leq E$ on $S$ such that
\begin{enumerate}
\item $D \cdot C \leq 2g\left(C\right)-2-C^2$ for any component $C$ of a fibre, strictly for a smooth fibre,
\item the inclusions $\caL_E \hra \Omega_{S/B}^1$ and $\omega_{S/B}\left(-D\right) \hra \omega_{S/B}$ fit into the following chain
\beqn
\caL_E \longhra \omega_{S/B}\left(-D\right) \longhra \Omega_{S/B}^1 \stackrel{\alpha}{\longhra} \omega_{S/B},
\enqn
\item the injection $\omega_{S/B}\left(-D\right) \hra \Omega_{S/B}^1$ lifts to $\Omega_S^1$, and
\item the quotient $\Omega_S^1 / \omega_{S/B}\left(-D\right)$ is isomorphic to
\beqn
f^*\omega_B \otimes \caO_S\left(D\right) \otimes I_{\Gamma}
\enqn
for some finite subscheme $\Gamma \subset S$, hence torsion-free.
\end{enumerate}
\end{proposition}
\begin{proof}
We proceed in two steps.
\begin{itemize}
\item[\textit{Step 1:}] \textit{$E$ satisfies 1, 2 and 3.}

We first show that the double dual $\caL_E^{\vee\vee}$ still injects into $\Omega_{S/B}^1$ and lifts to $\Omega_S^1$. Indeed, $\caL_E$ also injects in $\omega_{S/B}$ because $\alpha$ is injective. Hence the inclusions of $\caL_E$ into $\Omega_S^1$ and $\omega_{S/B}$ induce inclusions of $\caL_E^{\vee\vee}$ that fit into the commutative diagram
\beqn
\xymatrix{
 & \caL_E^{\vee\vee} \ar@{_(->}[]+<-1mm,-4mm>;[dl] \ar@{^(-->}[]+<0mm,-4mm>;[d] \ar@{^(->}[]+<3mm,-4mm>;[dr] & \\
\Omega_S^1 \ar@{->>}[r] & \Omega_{S/B}^1 \ar[r]^{\alpha}& \omega_{S/B}
}
\enqn
Therefore the composition $\caL_E^{\vee\vee} \hra \Omega_S^1 \ra \Omega_{S/B}^1$ must still be injective, as claimed, and it clearly lifts to $\Omega_S^1$ by construction. We have then a sequence of nested sheaves
\beqn
\caL_E \longhra \caL_E^{\vee\vee} \longhra \Omega_{S/B}^1 \stackrel{\alpha}{\longhra} \omega_{S/B},
\enqn
and we will be done if $\caL_E^{\vee\vee} \cong \omega_{S/B}\left(-E\right)$. In fact, completing the diagram with exact rows
\beqn
\xymatrix{
0 \ar[r] & \caL_E \ar[r] \ar@{-->}[d] & \Omega_{S/B}^1 \ar[r] \ar@{^(->}[]+<0mm,-4mm>;[d]^{\alpha} & \Omega_{S/B|E}^1 \ar[d] \ar[r] & 0 \\
0 \ar[r] & \omega_{S/B}\left(-E\right) \ar[r] & \omega_{S/B} \ar[r] & \omega_{S/B|E} & 
}
\enqn
we obtain an inclusion $\iota_E: \caL_E \hra \omega_{S/B}\left(-E\right)$ which is an isomorphism away from the critical points of $f$, and the same holds for its double dual $\iota_E^{\vee\vee}: \caL_E^{\vee\vee} \hra \omega_{S/B}\left(-E\right)$. But the critical points form a set of codimension 2 because $f$ has reduced fibres, hence $\iota_E^{\vee\vee}$ is an isomorphism, as wanted.

\item[\textit{Step 2:}] \textit{Removing the torsion of the cokernel.}

We have already obtained that $\omega_{S/B}\left(-E\right)$ lifts to $\Omega_S^1$. Denote by $\caM_0 \subseteq \Omega_S^1$ its image, and by $\widetilde{\caK}$ the quotient $\Omega_S^1/\caM_0$. Let $\caT$ be the torsion subsheaf of $\widetilde{\caK}$, and $\caK=\widetilde{\caK}/\caT$ its torsion-free quotient. Finally, let $\caM$ be the kernel of the composition of surjections $\Omega_S^1 \thra \widetilde{\caK} \thra \caK$.

We will now see that $\caM \cong \omega_{S/B}\left(-D\right)$ for some $0 \leq D \leq E$. Clearly, it is torsion-free, and the inclusion $\caM \hra \Omega_S^1$ factors as $\caM \hra \caM^{\vee\vee} \hra \Omega_S^1$. Consider the exact diagram
\beqn
\xymatrix{
 & & 0 \ar[d] & 0 \ar[d] & & \\
 & 0 \ar[r] & \caM \ar[r] \ar[d] & \caM^{\vee\vee} \ar[r] \ar[d] & \caG \ar[r] & 0 \\
 & 0 \ar[r] & \Omega_S^1 \ar@{=}[r] \ar[d] & \Omega_S^1 \ar[r] \ar[d] & 0 & \\
 0 \ar[r] & \caG \ar[r] & \caK \ar[r] \ar[d] & \caF \ar[r] \ar[d] & 0 \\
 & & 0 & 0 & &
}
\enqn
where we have used the snake lemma to identify the cokernel of the first row and the kernel of the last row. On the one hand, both $\caM$ and $\caM^{\vee\vee}$ have rank one, so $\caG$ is a torsion sheaf and, on the other hand, $\caG$ is torsion free since $\caK$ is. Therefore $\caG = 0$ and $\caM \cong \caM^{\vee\vee}$ is locally free. To finish, the composition $\caM \hra \Omega_S^1 \ra \omega_{S/B}$ is injective. Indeed, the image $\widetilde{\caM}$ is of rank 1 because $\caM_0 \subseteq \caM$ and the image of $\caM_0$ is $\omega_{S/B}\left(-E\right)$, so the kernel of $\caM \ra \omega_{S/B}$ is a rank-zero subsheaf of a torsion-free sheaf, hence zero. Therefore,
\beqn
\caM \cong \widetilde{\caM} = \omega_{S/B}\left(-D\right)
\enqn
with $D \leq E$ because by construction $\omega_{S/B}\left(-E\right) \subseteq \widetilde{\caM}$.

For the other assertion about $\caK = \Omega_S^1/\omega_{S/B}\left(-D\right)$, we first compute the Chern class
\beqn
c_1\left(\caK\right) = c_1\left(\Omega_S^1\right) - c_1\left(\omega_{S/B}\left(-D\right)\right) = c_1\left(f^*\omega_B \otimes \caO_S\left(D\right)\right).
\enqn
Since $\caK$ is torsion-free, this means that $\caK \cong f^*\omega_B \otimes \caO_S\left(D\right) \otimes L \otimes I_{\Gamma}$ for some finite subscheme $\Gamma \subset S$ and some $L \in \Pic^0\left(S\right)$. Consider now the diagram of exact rows
\beq \label{diag-aux}
\xymatrix{
0 \ar[r] & \omega_{S/B}\left(-D\right) \ar[r] \ar@{=}[d] & \Omega_S^1 \ar[d] \ar[r] & \caK \ar[r] \ar[d] & 0 \\
0 \ar[r] & \omega_{S/B}\left(-D\right) \ar[r] & \omega_{S/B} \ar[r] & \omega_{S/B|D} \ar[r] & 0 \\
}
\enq
Since $f$ has reduced fibres, the map $\alpha: \Omega_{S/B}^1 \ra \omega_{S/B}$ is injective and therefore the central map in (\ref{diag-aux}) has kernel $f^*\omega_B$ and cokernel $\omega_{S/B|Z'}$. The snake lemma leads then to the exact sequence
\beqn
0 \longra f^*\omega_B \longra \caK \longra \omega_{S/B|D} \longra \omega_{S/B|Z'} \longra 0.
\enqn
The first map corresponds to a section
\beqn
\sigma \in H^0\left(S,\caO_S\left(D\right) \otimes L \otimes I_Z\right) \subset H^0\left(S,\caO_S\left(D\right) \otimes L\right)
\enqn
whose zero scheme is $D$. Indeed, the zero scheme $Z\left(\sigma\right)$ is contained in $D$, and coincides with it outside the finite subscheme $Z'$. This implies that $L \cong \caO_S$ and we are done.
\end{itemize}
\end{proof}

\begin{remark}[About Step 2 in the proof of Proposition \ref{prop-conjunta}]
If a sheaf of the form $\omega_{S/B}\left(-E\right)$ lifts to $\caM_0 \subset \Omega_S^1$, there is a geometric interpretation of the support of the divisor $E$: it is the locus where $\caM_0 \subset \Omega_S^1$ is not transverse to $f^*\omega_B \subset \Omega_S^1$, that is
\begin{align*}
\Supp E & = \left\{p \tq \im\left(\left(f^*\omega_B \oplus \caM_0\right)_p \ra \Omega_{S,p}^1\right) \neq \Omega_{S,p}^1\right\} =\\
 & = \left\{p \tq \im\left(\left(f^*\omega_B \oplus \caM_0\right) \otimes \mC\left(p\right) \ra \Omega_S^1 \otimes \mC\left(p\right)\right) \neq \Omega_S^1 \otimes \mC\left(p\right)\right\}.
\end{align*}
The failure of the transversality at some regular point $p \in E'$ may occur either because
\begin{enumerate}
\item the images of $\caM_0\otimes\mC(p)$ and $\left(f^*\omega_B\right)\otimes\mC(p)$ in $\Omega_S^1 \otimes \mC\left(p\right)$ coincide, or
\item because $\caM_0\otimes\mC(p)$ maps to zero in $\Omega_S^1\otimes\mC(p)$.
\end{enumerate}
The first case means that not all local sections of $\caM_0$ vanish at $p$, but their values are proportional to pull-backs of 1-forms on $B$, while the second case means that all local sections of $\caM_0$ vanish at $p$. A computation in local coordinates shows that if the second case happens along some components $E_0$ of $E$, the quotient sheaf $\Omega_S^1/\caM_0$ would have torsion supported on $E_0$. The last step in the proof of Proposition \ref{prop-conjunta} replaces $E$ by $D=E-E_0$.
\end{remark}

We close now this section with Theorem \ref{main-thm} about the structure of fibrations supported on relatively rigid divisors (that is, divisors whose restriction to a general fibre is rigid).

\begin{theorem} \label{main-thm}
Let $S$ be a compact surface, and $f: S \ra B$ a stable fibration by curves of genus $g$ and relative irregularity $q_f =q\left(S\right) - g\left(B\right) \geq 2$. Suppose $f$ is supported on an effective divisor $D$ without components contained in fibres. Suppose also that $D \cdot C \leq 2g\left(C\right)-2-C^2$ for any component $C$ of a fibre, and that $h^0\left(F,\caO_F\left(D_{|F}\right)\right) = 1$ for some smooth fibre $F$. Then there is another fibration $h: S \ra B'$ over a curve of genus $g\left(B'\right) = q_f$. In particular $S$ is a covering of the product $B \times B'$, and both surfaces have the same irregularity.
\end{theorem}
\begin{proof}

Notice that if $g = 2$, then $g = q_f$ and $f$ is trivial according to the Lemma in the Appendix of \cite{Deb-Beau}. The statement is therefore immediate in this case, and we can assume from now on that $g>2$.

By Riemann-Roch, $h^0\left(F,\caO_F\left(D_{|F}\right)\right)=1$ implies $\deg \left(D_{|F}\right) = D \cdot F \leq g < 2g-2$. Therefore we can apply Corollary \ref{cor-converse-split} and hence the inclusion $\caL_D \hra \Omega_{S/B}^1$ lifts to $\Omega_S^1$. Applying Proposition \ref{prop-conjunta}, we can replace $D$ by a subdivisor (still called $D$ for simplicity) and assume that $\omega_{S/B}\left(-D\right)$ lifts to $\Omega_S^1$ and that the cokernel $\caK = \caK_D$ of the lifting is torsion-free, isomorphic to $f^*\omega_B \otimes \caO_S\left(D\right) \otimes I_{\Gamma}$ for some finite subscheme $\Gamma \subset S$. Since we have replaced $D$ by a subdivisor, it still holds that $h^0\left(F,\caO_F\left(D_{|F}\right)\right) = 1$ for some smooth fibre $F$.

\noindent\textit{Claim:} $h^0\left(S,\omega_{S/B}\left(-D\right)\right) \geq q_f$. Indeed, it follows from the exact sequence
\beqn
0 \longra \omega_{S/B}\left(-D\right) \longra \Omega_S^1 \longra \caK \longra 0
\enqn
that $h^0\left(\omega_{S/B}\left(-D\right)\right) \geq h^0\left(\Omega_S^1\right) - h^0\left(\caK\right) = q\left(S\right) - h^0\left(\caK\right)$, so it is enough to prove that $h^0\left(\caK\right) = g\left(B\right)$.

Since $f$ has reduced fibres, $\Omega_{S/B}^1$ is a subsheaf of $\omega_{S/B}$ (Lemma \ref{lem-versus}) and the sequence
\beqn
0 \longra f^*\omega_B \longra \Omega_S^1 \longra \omega_{S/B}
\enqn
is exact. Applying the snake lemma to the diagram of exact rows
\beqn
\xymatrix{
0 \ar[r] & \omega_{S/B}(-D) \ar[r] \ar@{=}[d] & \Omega_S^1 \ar[r] \ar[d] & \caK \ar[r] \ar[d] & 0 \\
0 \ar[r] & \omega_{S/B}(-D) \ar[r] & \omega_{S/B} \ar[r] & \omega_{S/B|D} \ar[r] & 0
}
\enqn
we get that the kernel of $\caK \ra \omega_{S/B|D}$ is also $f^*\omega_B$. Therefore, taking direct images, we obtain the following exact sequence of sheaves on $B$
\beqn
0 \longra \omega_B \longra f_*\caK \longra f_*\omega_{S/B|D}.
\enqn

Since $\caK \cong f^*\omega_B \otimes \caO_S\left(D\right) \otimes I_{\Gamma}$ is torsion-free and $D_{|F}$ is rigid for a general fibre $F$,
\beqn
f_*\caK = \omega_B \otimes f_*\left(\caO_S(D)\otimes I_{\Gamma}\right)
\enqn
is a vector bundle (torsion-free over a curve) of rank one. Therefore, the cokernel of $\omega_B \hookrightarrow f_*\caK$ must be a torsion subsheaf of $f_*\omega_{S/B|D}$. But the latter is torsion-free because $D$ has no component contracted by $f$ (see Lemma \ref{lem-final-2.4} below), so the injection $\omega_B \hookrightarrow f_*\caK$ is in fact an isomorphism, and
\beqn
h^0(S,\caK) = h^0(B,f_*\caK) = h^0(B,\omega_B) = g(B),
\enqn
finishing the proof of the claim.

Since the lifting of $\omega_{S/B}\left(-D\right)$ to $\Omega_S^1$ is a line bundle $\caL$, the wedge product of any two of its sections is zero. Therefore, since we have just seen that $h^0\left(\caL\right) \geq q_f \geq 2$, the Castelnuovo-de Franchis Theorem (\cite{Cat}, Theorem 1.9) implies the existence of the fibration $h: S \ra B'$ over a curve $B'$ of genus $g\left(B'\right) \geq q_f$.

It remains to show that $g\left(B'\right) = q_f$, which follows from the last structural statement. In fact, the two fibrations give a covering $\pi$ completing the diagram
\beqn
\xymatrix{
 & S \ar[dl]_f \ar[dr]^h \ar@{-->}[d]^{\pi}& \\
B & B \times B' \ar[l] \ar[r] & B'
}
\enqn
Since $\pi$ is surjective, $q \left(S\right) \geq q\left(B\times B'\right) = g\left(B\right) + g\left(B'\right)$, hence $g\left(B'\right) \leq q_f$, and the proof is finished.
\end{proof}

The proof of the following Lemma is immediate and is left to the reader.

\begin{lemma} \label{lem-final-2.4}
If $f: S \ra B$ is any fibration, $D$ is an effective divisor on $S$ without components contracted by $f$, and $L$ is any line bundle on $S$, then $f_* \left(L_{|D}\right)$ is a torsion-free sheaf on $B$.
\end{lemma}


\section{The global adjoint map and supporting divisors}

The main topic of this Section are {\em adjoint images}, which have proved to be a useful tool to study both infinitesimal and local deformations of irregular varieties. They were introduced in the study of curves by Collino and Pirola in \cite{ColPir}, and then extended to higher-dimensional varieties by Pirola an Zucconi in \cite{PZ}. A more intrinsical construction, in terms of {\em adjoint maps}, was given later by Pirola and Rizzi in \cite{Pir-Riz} for smooth families of curves. The aim of this section is to generalize the construction of the adjoint map both to first-order deformations of higher-dimensional (irregular) varieties, and to all the fibres of a family of curves (in particular, to fibred surfaces). This second generalization will allow us to produce supporting divisors under suitable assumptions on the fibration.

\subsection{Adjoint map of an infinitesimal deformation}

We first introduce adjoint maps. Although our main applications deal with curves, most of the constructions and basic results also work for higher dimensions. Hence, we present adjoint maps in their most general form, for varieties of arbitrary dimension. 

Let $X$ be a smooth projective variety of dimension $d$. For any integer $1 \leq k \leq d$ we consider the map
\beqn
\psi_k: \bigwedge^k H^0\left(X,\Omega_X^1\right) \longra H^0\left(X,\Omega_X^k\right)
\enqn
given by wedge product. Given a subspace $W \subseteq H^0\left(X,\Omega_X^1\right)$, we define
\beqn
W^k = \psi_k\left(\bigwedge^k W\right) \subseteq H^0\left(X,\Omega_X^k\right).
\enqn
In particular, for $k=d$, we have $W^d \subseteq H^0\left(X,\omega_X\right)$.

\begin{definition}
If $W^d \neq 0$, denote by $D_W$ the base divisor of the linear series $\left|W^d\right| \subseteq \left|\omega_X\right|$.
\end{definition}

Consider now a first-order infinitesimal deformation $\caX \ra \Delta=\Spec\mC\left[\epsilon\right]/\left(\epsilon^2\right)$ of $X$, corresponding to an extension class $\xi \in \Ext_{\caO_X}^1\left(\Omega_X^1,\caO_X \otimes T_{\Delta,0}^{\vee}\right) \cong H^1\left(X,T_X\right) \otimes T_{\Delta,0}^{\vee}$, the Kodaira-Spencer class of the sequence of vector bundles
\beqn
0 \longra N_{X/\caX}^{\vee} = \caO_X \otimes T_{\Delta,0}^{\vee} \longra \Omega_{\caX|X}^1 \longra \Omega_X^1 \longra 0.
\enqn

The corresponding connecting homomorphism
\beqn
\partial_{\xi} = \cup\,\xi: H^0\left(X,\Omega_X^1\right) \longra H^1\left(X,\caO_X\right) \otimes T_{\Delta,0}^{\vee}
\enqn
is given by cup-product with $\xi$. Denote by
\beqn
K = K_{\xi} = \ker \partial_{\xi} = \im\left(H^0\left(X,\Omega_{\caX|X}^1\right) \longra H^0\left(X,\Omega_X^1\right)\right)
\enqn
the subspace of 1-forms on $X$ that can be extended to $\caX$, and assume $\dim K_{\xi} \geq d+1$ (in particular, $q\left(X\right) \geq d+1$).

Given any subspace $W \subseteq K_{\xi}$, denote by $\widetilde{W} \subseteq H^0\left(X,\Omega_{\caX|X}^1\right)$ its preimage, so that we have the following exact sequence
\beqn
0 \longra T_{\Delta,0}^{\vee} \longra \widetilde{W} \longra W \longra 0,
\enqn
from which we obtain the presentation
\beqn
T_{\Delta,0}^{\vee} \otimes \bigwedge^d \widetilde{W} \stackrel{\wedge}{\longra} \bigwedge^{d+1} \widetilde{W} \longra \bigwedge^{d+1} W \longra 0.
\enqn
Wedge product induces also a map
\beqn
\bigwedge^{d+1} \widetilde{W} \longra H^0\left(X,\Omega_{\caX|X}^{d+1}\right) \cong T_{\Delta,0}^{\vee} \otimes H^0\left(X,\omega_X\right),
\enqn
and the image of $T_{\Delta,0}^{\vee} \otimes \bigwedge^d \widetilde{W}$ maps onto $T_{\Delta,0}^{\vee} \otimes W^d$. Hence, there is a well-defined map
\beq \label{adj-map-1}
\nu_W: \, \bigwedge^{d+1} W \longra T_{\Delta,0}^{\vee} \otimes \left(H^0\left(X,\omega_X\right)/W^d\right)
\enq
completing the diagram below.
\beq \label{diag-def-nu_W}
\xymatrix{
T_{\Delta,0}^{\vee} \otimes \bigwedge^d \widetilde{W} \ar[r] \ar@{->>}[d] & \bigwedge^{d+1} \widetilde{W} \ar[r] \ar[d] & \bigwedge^{d+1} W \ar[r] \ar@{-->}[d]^{\nu_W} & 0 \\
T_{\Delta,0}^{\vee} \otimes W^d \ar@{^(->}[]+<10mm,0mm>;[r] & T_{\Delta,0}^{\vee} \otimes H^0\left(X,\omega_X\right) \ar[r] & T_{\Delta,0}^{\vee} \otimes \left(H^0\left(X,\omega_X\right)/W^d\right) \ar[r] & 0
}
\enq

\begin{definition}
The map $\nu_W$ in (\ref{adj-map-1}) is the {\em adjoint map associated to} $W$.
\end{definition}

In their works \cite{ColPir}, \cite{PZ}, Collino, Pirola and Zucconi restrict to the case $\dim W = d+1$, while the above construction works for any subspace $W$ of dimension at least $d+1$ (in fact, if $\dim W \leq d$, then $\bigwedge^{d+1}W = 0$ and $\nu_W$ is the zero map). They also start from a basis $\eta_1,\ldots,\eta_{d+1}$ of $W$, choose arbitrary preimages $s_1,\ldots,s_{d+1} \in H^0\left(X,\Omega_{\caX|X}^1\right)$, and show that the class $\left[w\right]$ of $s_1\wedge\cdots\wedge s_{d+1}$ in $H^0\left(X,\omega_S\right)/W^d$ is well-defined up to scalar (depending on a chosen isomorphism $T_{\Delta,0} \cong \mC$ and the basis $\left\{\eta_i\right\}$). They call $\left[w\right]$ {\em an} adjoint class of $W$, which in our setting is the image of $\eta_1\wedge\cdots\wedge\eta_{d+1}$ by the adjoint map $\nu_W$.

We can now state one of the most powerful results about adjoint images (or maps): the Adjoint Theorem. It was first proven by Collino and Pirola for curves (\cite{ColPir} Th. 1.1.8), and then it was generalized to arbitrary dimensions by Pirola and Zucconi (\cite{PZ} Th. 1.5.1). Recall that the deformation $\xi$ is said to be supported on an effective divisor $D \subset X$ if 
\beqn
\xi \in \ker\left(H^1\left(X,T_X\right) \longra H^1\left(X,T_X\left(D\right)\right)\right).
\enqn

\begin{theorem}[Adjoint Theorem, \cite{PZ} Thm.1.5.1, \cite{ColPir} Thm.1.1.8 for curves] \label{thm-adj}
Let $W \subseteq K_{\xi} \subseteq H^0(X,\omega_X)$ be a $\left(d+1\right)$-dimensional subspace such that $W^d \neq 0$, and let $D=D_W$ be the base locus of the corresponding linear series $|W^d| \subseteq \left|\omega_X\right|$. If the adjoint map of $W$ is the zero map, then $\xi$ is supported on $D$.
\end{theorem}

Our next objective is to use Theorem \ref{thm-adj} to produce supporting divisors for a given deformation $\xi$. In particular, we want to give a condition on $\dim K_{\xi}$ that guarantees the existence of a $(d+1)$-dimensional subspace $W$ with vanishing adjoint map. Some of the following constructions and results are inspired by the study of special deformations carried out by Collino and Pirola in \cite{ColPir}, Section 1.3.

We will need the following

\begin{ass} \label{extra-ass}
For any $(d+1)$-dimensional subspace, $\bigwedge^d W$ {\em injects} into $H^0\left(X,\omega_X\right)$. According to the Generalized Castelnuovo-de Franchis Theorem (\cite{Cat}, Theorem 1.9), this can be more geometrically rephrased as ``$X$ does not admit higher-irrational pencils'' (i.e. fibrations over varieties whose Albanese map is birational and not surjective).
\end{ass}

Since this assumption holds automatically if $X = C$ is a curve, and our main applications are for curves, this is not really a restrictive condition for our purposes. Furthermore, the condition could be relaxed to ``the images $W^d \subseteq H^0\left(X,\omega_X\right)$ have all the same dimension''.

Let $\mG$ be the Grassmannian of $(d+1)$-dimensional subspaces of $K=K_{\xi}$. For any vector space $E$, denote by $E_{\mG} = E \otimes \caO_{\mG}$ the trivial vector bundle on $\mG$ with fibre $E$. As customary, denote by $\caS \subseteq K_{\mG}$ and $\caQ = K_{\mG}/S$ the tautological subbundle and quotient bundle. Note that the extra assumption above implies that $\bigwedge^d \caS$ injects in $H^0\left(X,\omega_X\right)_{\mG}$ as a vector bundle of rank $d+1$, and the quotient is also a vector bundle (of rank $p_g\left(X\right)-(d+1)$).

\begin{lemma} \label{lem-adj-holom}
The adjoint maps $\nu_W$ depend holomorphically on $W \in \mG$. More precisely, there exists a map of vector bundles
\beqn
\nu: \bigwedge^{d+1} \caS \longra T_{\Delta,0}^{\vee} \otimes \left(H^0\left(X,\omega_X\right)_{\mG}\left/\bigwedge^d \caS\right.\right).
\enqn
such that $\nu\otimes\mC\left(W\right) = \nu_W$.
\end{lemma}
\begin{proof}
The proof is immediate. One only has to mimick the construction of the $\nu_W$ replacing $W$ by the tautological subbundle $\caS$.
\end{proof}

\begin{definition} \label{df-adj-bundle}
We call the map $\nu$ constructed in the previous Lemma simply the {\em adjoint map} of the deformation $\xi$. It can be seen as a section of the vector bundle
\beqn
\caA = T_{\Delta,0}^{\vee} \otimes \bigwedge^{d+1} \caS^{\vee} \otimes \left(H^0\left(X,\omega_X\right)_{\mG}\left/\bigwedge^d \caS\right.\right),
\enqn
which we call the {\em adjoint bundle}.
\end{definition}

Computing the Chern class of $\caA$ leads to the proof of the following

\begin{theorem} \label{thm-adj-inft}
If $V \subseteq K_{\xi}$ has dimension $\dim V \geq \frac{p_g\left(X\right)}{d+1}+d$, then there exists some $(d+1)$-dimensional subspace $W \subseteq V$ such that $\nu_W = 0$.
\end{theorem}
\begin{proof}
Denote by $\mG_V = Gr\left(d+1,V\right) \subseteq \mG$ the subvariety of $\mG$ consisting of the $(d+1)$-dimensional subspaces of $K$ contained in $V$, which is in turn a Grassmannian variety. Furthermore, the tautological subbundle $\caS_V$ of $\mG_V$ is the restriction of $\caS$, and the adjoint map $\nu$ restricts to
\beqn
\nu_V: \, \bigwedge^{d+1} \caS_V \longra T_{\Delta,0}^{\vee} \otimes \left(H^0\left(X,\omega_X\right)_{\mG_V}\left/\bigwedge^d\caS_V\right.\right)
\enqn
which is a section of the vector bundle
\beqn
\caA_V = T_{\Delta,0}^{\vee} \otimes \bigwedge^{d+1} \caS_V^{\vee} \otimes \left(H^0\left(X,\omega_X\right)_{\mG_V}\left/\bigwedge^d\caS_V\right.\right) = \caA_{|V}.
\enqn
Denoting by $Z = Z\left(\nu\right) \subseteq \mG$ the zero locus of $\nu$, and by $Z_V$ the zero locus of $\nu_V$, it is clear that $Z_V = Z \cap \mG_V$.

With these notations, the theorem says that $Z_V \neq \emptyset$. In order to prove that, we will compute the top Chern class of $\caA_V$ and show that it does not vanish. This is enough, since if a vector bundle admits a nowhere vanishing section, then its top Chern class is zero.

First of all, our only hypothesis is equivalent to
\beqn
r = \rk \caA_V = p_g\left(X\right)-(d+1) \leq (d+1)\left(\dim V - (d+1)\right) = \dim \mG_V,
\enqn
so it is indeed possible that $c_r\left(\caA_V\right) \neq 0$. Secondly, up to the trivial twisting by $T_{\Delta,0}^{\vee}$, $\caA_V$ is the globally generated bundle
\beqn
\caG = H^0\left(X,\omega_X\right)_{\mG_V}\left/\bigwedge^d\caS_V\right.
\enqn
twisted by the line bundle $\bigwedge^{d+1} \caS_V^{\vee} \cong \caO_{\mG_V}\left(1\right)$, the very ample line bundle inducing the Pl\"ucker embedding. Therefore, we can compute (see \cite{Fulton} Remark 3.2.3.(b))
\beqn
c_r\left(\caA_V\right) = \sum_{i=0}^r c_{r-i}\left(\caG\right)c_1\left(\caO_{\mG_V}\left(1\right)\right)^i = c_1\left(\caO_{\mG_V}\left(1\right)\right)^r + \left(\text{effective classes}\right) \neq 0
\enqn
because $\caO_{\mG_V}\left(1\right)$ is very ample and all the Chern classes of $\caG$ are represented by zero or effective cycles (because it is globally generated).
\end{proof}

\begin{corollary} \label{cor-adj-inft}
If $X=C$ is a curve of genus $g$, and $V \subseteq K_{\xi}$ has dimension greater than $\frac{g+1}{2}$, then there exists a two-dimensional subspace $W \subseteq V$ whose adjoint class vanishes. In particular, the deformation is supported on a divisor $D$ of degree $\deg D < 2g-2$.
\end{corollary}
\begin{proof}
The first assertion follows directly from Theorem \ref{thm-adj-inft}, taking $d=1$. For the second assertion, take $D$ the base divisor of the pencil $\left|W\right| \subseteq \left|\omega_C\right|$.
\end{proof}

\begin{remark}
For higher dimensions, the inequality $\dim V \geq \frac{p_g\left(X\right)}{d+1}+d$, combined with the non-existence of higher irrational pencils (assumption \ref{extra-ass}), becomes a quite restrictive condition. For example, the only surfaces to which this method could be applied are those satisfying
\beqn
2q\left(X\right)-3 \leq p_g\left(X\right) \leq 3\left(q\left(X\right)-2\right),
\enqn
where the first inequality is the Castelnuovo-de Franchis inequality.
\end{remark}


\subsection{Global adjoint map}

In this last section we extend the previous constructions to the case of a fibration over a compact curve. Since the extra assumption \ref{extra-ass} is quite restrictive, we will only consider the case when the fibres are curves (though some constructions carry over to some cases with higher-dimensional fibres).

Therefore, let $f: S \ra B$ be a fibration of a surface $S$ over a compact curve $B$, and let
\beqn
V = V_f = H^0\left(S,\Omega_S^1\right) / f^*H^0\left(B,\omega_B\right),
\enqn
which has dimension $q_f$, the {\em relative irregularity} of $f$. It is easy to see that $V$ naturally injects into $H^0\left(F,\omega_F\right)$ for any smooth fibre $F$ of $f$. Furthermore, if $\xi \in H^1\left(F,T_F\right)$ is the infinitesimal deformation of $F$ induced by $f$, then $V$ is contained in the kernel $K_{\xi}$
\beqn
\partial_{\xi} = \cup\,\xi: H^0\left(F,\omega_F\right) \longra H^1\left(F,\caO_F\right).
\enqn

We have previously constructed the adjoint map associated to any subspace of $K_{\xi}$. In order to obtain a ``global'' adjoint map valid for all the fibres, we restrict now to a slightly less general version, considering only subspaces $W$ of $V$.

All the injections $V \subseteq H^0\left(F,\omega_F\right)$ for smooth fibres glue together into an inclusion of vector bundles
\beq \label{eq-incl-V_B}
V_B = V \otimes \caO_B \longhra f_* \omega_{S/B}
\enq
whose cokernel $\caG$ is also a vector bundle (see \cite{Fuj} Theorem 3.1 and its proof). In fact, the results of Fujita say moreover that the inclusion splits (so $f_* \omega_{S/B} \cong V_B \oplus \caG$) and $\caG$ has some good cohomological properties, but we will not use them in the sequel.

The inclusion (\ref{eq-incl-V_B}) can be more explicitly constructed as follows (note the similarities with the definition of $\alpha$ in Lemma \ref{lem-versus}). Wedge product gives a map $H^0\left(S,\Omega_S^1\right) \otimes \omega_B \ra f_*\omega_S$ mapping $\left(f^*H^0\left(B,\omega_B\right)\right) \otimes \omega_B$ to zero, so there is an induced map
\beqn
V \otimes \omega_B \longra f_*\omega_S = \left(f_* \omega_{S/B}\right) \otimes \omega_B.
\enqn
Since it is injective over a generic $b \in B$, it is everywhere injective (as a map of sheaves), and cancelling the twist by $\omega_B$ we obtain the inclusion (\ref{eq-incl-V_B}).

Denote now by $\mG = Gr\left(2,V\right)$ the Grassmannian of 2-planes of $V$, and by $S_V \subseteq V \otimes \caO_{\mG}$ the tautological subbundle. Consider the product $Y = B \times \mG$, and denote by $p_1: Y \ra B$ and $p_2: Y \ra \mG$ the natural projections. The variety $Y$ is the Grassmann bundle of 2-dimensional subspaces of $V_B$, and $\caS = p_2^* S_V$ is the corresponding tautological subbundle. Clearly, $\caS$ is a vector subbundle\footnote{By a {\em vector subbundle} of a vector bundle $V$ we mean a locally free subsheaf whose quotient is also locally free.} of $V_Y = V \otimes \caO_Y = p_1^* V_B$, hence also of $p_1^*f_*\omega_{S/B}$.

We will now reproduce the construction of the adjoint map for an infinitesimal deformation. Denote by $\widetilde{\caS} \subseteq H^0\left(S,\Omega_S^1\right) \otimes \caO_Y$ the natural preimage of $\caS$, so that
\beqn
0 \longra H^0\left(B,\omega_B\right) \otimes \caO_Y \stackrel{f^*}{\longra} \widetilde{\caS} \longra \caS \longra 0
\enqn
is an exact sequence of vector bundles. Hence we obtain the following presentation of $\bigwedge^2 \caS$,
\beq \label{global-presentation}
\widetilde{\caS} \otimes H^0\left(B,\omega_B\right) \longra \bigwedge^2 \widetilde{\caS} \longra \bigwedge^2 \caS \longra 0
\enq

The composition of the wedge product $\bigwedge^2 H^0\left(S,\Omega_S^1\right) \ra H^0\left(S,\omega_S\right) = H^0\left(B,f_*\omega_S\right)$ and the evaluation $H^0\left(B,f_*\omega_S\right) \otimes \caO_Y \cong H^0\left(Y,p_1^*f_*\omega_S\right) \otimes \caO_Y \ra p_1^*f_*\omega_S$ induces a map of vector bundles on $Y$
\beqn
\widetilde{\nu}: \, \bigwedge^2 \widetilde{\caS} \longra p_1^*f_*\omega_S.
\enqn
Clearly, this map sends the image of $\widetilde{\caS} \otimes H^0\left(B,\omega_B\right)$ into the subsheaf $\caS \otimes p_1^*\omega_B$. Hence, according to equation (\ref{global-presentation}), $\widetilde{\nu}$ induces a well-defined map of vector bundles on $Y$:
\beq \label{eq-global-adjoint}
\nu: \, \bigwedge^2 \caS \longra \left(p_1^*f_*\omega_S\right)/\left(\caS \otimes p_1^*\omega_B\right).
\enq

\begin{definition}[Global Adjoint Map] \label{df-global-adjoint}
The map $\nu$ in (\ref{eq-global-adjoint}) is the {\em global adjoint map} of the fibration $f$.
\end{definition}

\begin{remark} \label{rmk-global-adj}
It is clear from the construction that if $F=f^{-1}\left(b\right)$ is a smooth fibre of $f$, the restriction $\nu_{|\left\{b\right\}\times\mG}$ coincides with the adjoint map constructed in Definition \ref{df-adj-bundle}, restricted to the Grassmannian subvariety $Gr\left(2,V\right)$.
\end{remark}

To close both this section and the article, we will combine the global adjoint map with Corollary \ref{cor-adj-inft}, considering only vector subbundles of rank two $\caW \subseteq V \otimes \caO_B$. Such a vector subbundle defines a section
\beqn
\eta_{\caW}: B \longra Y
\enqn
of $p_1$, such that $\eta_{\caW}\left(b\right) = \caW \otimes \mC\left(b\right) \subseteq V$. Conversely, given any section $\eta: B \ra Y$ of $p_1$, it defines the vector subbundle
\beqn
\caW_{\eta} = \eta^*\caS \longhra \eta^*\left(V \otimes \caO_Y\right) = V \otimes \caO_B.
\enqn
Clearly, the assignations $\caW \mapsto \eta_{\caW}$ and $\eta \mapsto \caW_{\eta}$ are mutually inverse, giving a one-to-one correspondence between vector subbundles of $V\otimes \caO_B$ of rank 2 and sections of $p_1: Y \ra B$.

Now, given a vector subbundle $\caW$ as above, we can consider the restriction $\nu_{\caW}$ of the adjoint map $\nu$ to the curve $\eta_{\caW}\left(B\right) \cong B$, which can be seen as a map of vector bundles on $B$:
\beq \label{eq-global-adj-subb}
\nu_{\caW}: \bigwedge^2 \caW \longra \left(f_*\omega_S\right)/\left(\caW \otimes \omega_B\right).
\enq

\begin{definition}[Global Adjoint Map associated to a subbundle] \label{df-global-adj-subb}
We call the map $\nu_{\caW}$ in equation (\ref{eq-global-adj-subb}) the {\em global adjoint map associated to} the subbundle $\caW$.
\end{definition}

We are now ready to state the wanted global result.

\begin{theorem} \label{pro-adj-global}
If
\beqn
q_f > \frac{g+1}{2},
\enqn
then there exist a base change $\pi: B' \ra B$ and a rank-two vector subbundle $\caW \subseteq V \otimes \caO_{B'}$ whose associated global adjoint map vanishes identically.
\end{theorem}
\begin{proof}
Let $Z \subseteq Y$ be the zero set of the global adjoint map $\nu$, which is an analytic subvariety. By Remark \ref{rmk-global-adj}, for any regular value $b$, the set $Z_b = Z \cap \left(\left\{b\right\}\times\mG\right)$ is the vanishing set of the adjoint map of $F_b$, which is non-empty by Corollary \ref{cor-adj-inft}. Therefore, there is a component of $Z$ dominating $B$, hence it is possible to choose an irreducible curve $\widetilde{B} \subseteq Z$ dominating $B$. Let $\mu: B' \ra \widetilde{B}$ be the normalization of $\widetilde{B}$, and define $\pi$ as the composition $p_1 \circ \mu: B' \ra B$. As for the vector subbundle, let $\eta: B' \ra B' \times_B Y \cong B' \times \mG$ be the section induced from the map $B' \ra \widetilde{B} \ra Y$, and let $\caW = \caW_{\eta}$. Since the image of $\eta$ is contained in the zero locus of the adjoint map associated to the fibration $S' = S \times_B B' \ra B'$, (see next Remark), it is tautological that the global adjoint map associated to $\caW$ vanishes identically.
\end{proof}

\begin{remark} [Global Adjoint Maps and base change]
Consider a finite morphism $\pi: B' \ra B$, let $f': S' = \widetilde{S \times_B B'} \ra B'$ be the fibration obtained after change of base and desingularization, and $V'=V_{f'} = H^0\left(S',\Omega_{S'}^1\right)/\left(f'\right)^*H^0\left(B',\omega_B'\right)$ the corresponding space of relative 1-forms. Define also $\mG' = Gr\left(2,V'\right)$ and $Y' = B' \times \mG'$, and let $\nu'$ be the global adjoint map of $f'$.

Clearly $V$ injects into $V'$, and therefore $B' \times \mG$ is naturally a subvariety of $Y'$. Furthermore, the pull-back (by $\pi \times \id_{\mG}$) of $\nu$ is the restriction of $\nu'$ to $B' \times \mG$. Hence, the zero locus of $\nu'$ contains the preimage of the zero locus of $\nu$.
\end{remark}

\begin{corollary} \label{cor-adj-global}
If $f: S \ra B$ is a fibration of genus $g$ such that $q_f > \frac{g+1}{2}$, then after a base change as above, $f'$ is supported on a divisor $D$ such that $D \cdot F < 2g-2$ for any fibre $F$. Furthermore, if $f$ is relatively minimal with reduced fibres, then $D \cdot C \leq 2g\left(C\right)-2-C^2$ for any component $C$ of a fibre.
\end{corollary}
\begin{proof}
Let $\caW \subseteq V \otimes \caO_{B'} \subseteq f_*\omega_{S/B}$ be the rank-two vector subbundle provided by Theorem \ref{pro-adj-global}, and consider the relative evaluation map
\beqn
f^*\caW \longra f^*f_*\omega_{S/B} \longra \omega_{S/B}.
\enqn
Then it is enough to take $D$ as the union of the divisorial components of its base locus that dominate $B$. The inequalities for the intersection products $D \cdot C$ follow because $D$ is the base locus of a pencil of differential forms on any fibre.
\end{proof}


\section{Appendix: Relative $\caExt$ sheaves}

\label{rel_ext_sheaves}

Since they play a central role in Section \ref{sect-deform}, we include here a summary of the definition and some of the main properties of the relative Ext sheaves. They were originally introduced by Grothendieck in SGA2 \cite{SGA2}, and a slightly more general exposition can be found in the first chapter of \cite{Birkar}. Some extra properties are explicitly written down in \cite{Lan-Ext}.

\begin{definition}[Relative ext sheaves, \cite{Birkar} Def. 1.1.1]
Given a morphism of schemes (or more generally, of ringed spaces) $f: X \ra Y$, and an $\caO_X$-module $\caF$, we define $\caExt_f^p\left(\caF,-\right)$ as the $p$-th right derived functor of the left-exact functor $f_*\caHom_{\caO_X}\left(\caF,-\right)$.
\end{definition}

\begin{example}[\cite{Birkar}, Def.-Remark 1.1.2]
Some particular cases:
\begin{enumerate}
\item If $Y = \Spec \mC$ is a point, then $\caExt_f^p\left(\caF,-\right) = \Ext_{\caO_X}^p\left(\caF,-\right)$, the global $\Ext$ functor. If furthermore $\caF = \caO_X$, $\caExt_f^p\left(\caO_X,-\right) = H^p\left(X,-\right)$ is the usual sheaf cohomology.

\item If $f$ is the identity (hence $Y=X$), then $\caExt_f^p\left(\caF,-\right) = \caExt_{\caO_X}^p\left(\caF,-\right)$ is the usual local $\caExt$ functor.

\item If $\caF = \caO_X$, then $f_*\caHom_{\caO_X}\left(\caO_X,-\right) = f_*$ is the usual push-forward functor, and therefore $\caExt_f^p\left(\caO_X,-\right) = R^pf_*$ are the higher-direct image functors.
\end{enumerate}
\end{example}

\begin{theorem} \label{thm-ext}
Some properties:
\begin{enumerate}
\item (\cite{Birkar} Th. 1.1.3) For any $\caO_X$-modules $\caF,\caG$, $\caExt_f\left(\caF,\caG\right)$ is the sheaf associated to the presheaf
\beqn
U \mapsto \Ext_{\caO_{f^{-1}\left(U\right)}}^p\left(\caF_{|f^{-1}\left(U\right)},\caG_{|f^{-1}\left(U\right)}\right).
\enqn
In particular, for any open subset $W \subseteq Y$,
\beqn
\caExt_f^p\left(\caF,\caG\right)_{|f^{-1}\left(W\right)} \cong \caExt_f^p\left(\caF_{|f^{-1}\left(W\right)},\caG_{|f^{-1}\left(W\right)}\right).
\enqn

\item (\cite{Birkar} Th. 1.1.4) If $\caL$ and $\caN$ are locally free sheaves of finite rank on $X$ and $Y$, respectively, then
\beqn
\caExt_f^p\left(\caF \otimes \caL,-\otimes f^*\caN\right) \cong \caExt_f^p\left(\caF,-\otimes\caL^{\vee}\otimes f^*\caN\right) \cong \caExt_f^p\left(\caF,-\otimes\caL^{\vee}\right)\otimes\caN.
\enqn

\item (\cite{Birkar} Th. 1.1.5) If $0 \ra \caF' \ra \caF \ra \caF'' \ra 0$ is an exact sequence of $\caO_X$-modules, and $\caG$ is another $\caO_X$-module, then there is a long exact sequence
\begin{multline*}
\cdots \longra \caExt_f^{p-1}\left(\caF',\caG\right) \longra \\
\longra \caExt_f^p\left(\caF'',\caG\right) \longra \caExt_f^p\left(\caF,\caG\right) \longra \caExt_f^p\left(\caF',\caG\right) \longra \\
\longra \caExt_f^{p+1}\left(\caF'',\caG\right) \longra \cdots
\end{multline*}

\item (Local-global spectral sequence, \cite{Birkar} Th. 1.2.1) Suppose $g: Y \ra Z$ is another morphism, and denote $h=g \circ f$. For any $\caO_X$-modules $\caF,\caG$ there is a spectral sequence
\beqn
E_2^{p,q}=R^pg_*\caExt_f^q\left(\caF,\caG\right) \Ra \caExt_h^{p+q}\left(\caF,\caG\right).
\enqn

\item (Coherence, \cite{Birkar} Th. 1.3.1) If $f$ is projective and $\caF$, $\caG$ are coherent, then $\caExt_f^p\left(\caF,\caG\right)$ is coherent.
\end{enumerate}
\end{theorem}

The $\caExt_f^i$ sheaves being a generalization of the $R^if_*$, they also verify a base-change theorem whose proof relies on the same techniques. More explicitly, for any $y \in Y$ there is a natural base-change map
$$\tau^i\left(y\right): \caExt_f^i\left(\caF,\caG\right) \otimes \mC(y) \ra \Ext^i_{\caO_{X_y}}\left(\caF_{|X_y},\caG_{|X_y}\right)$$
which associates to a local extension $\xi \in \Ext_{\caO_{f^{-1}\left(U\right)}}^i\left(\caF_{|f^{-1}\left(U\right)},\caG_{|f^{-1}\left(U\right)}\right)$ around $y$, the class of the restriction to $X_y$. Then the following result holds:

\begin{theorem}[Base-change theorem for $\caExt_f$, \cite{Lan-Ext}] \label{thm-base-change-ext}
Assume $\caF$ and $\caG$ are flat over $Y$. Let $y \in Y$ be a point and assume the base change homomorphism $\tau^i\left(y\right)$ to be surjective.
Then
\begin{enumerate}
\item there is a neighbourhood $U$ of $y$ such that $\tau^i\left(y'\right)$ is an isomorphism for all $y' \in U$, and
\item $\tau^{i-1}\left(y\right)$ is surjective if and only if $\caExt_f^i\left(\caF,\caG\right)$ is locally free in a neighbourhood of $y$.
\end{enumerate}
\end{theorem}

{\bf Acknowledgements:}
I would like to thank Prof. Gian Pietro Pirola for the very fruitful discussions about his previous works related to this topic, and Prof. S\l awomir Rams for his ideas to fix some last minute technical details. I would also like to thank the referee for his/her several suggestions to simplify and improve the presentation of this work. Since this work is a big part of my PhD thesis, I am also especially grateful to my advisors Miguel \'Angel Barja and Juan Carlos Naranjo, because this work would not have been possible without them.


\end{document}